\documentclass[11pt]{amsart}

\usepackage{amsfonts,amsmath}
\usepackage{amssymb, latexsym}
\usepackage{amscd,amsthm}
\usepackage{enumitem}
\usepackage[all]{xy}
\usepackage{srcltx} %inverse search

\newcommand{\marginlabel}[1]%
 {\mbox{}\marginpar{\raggedleft\hspace{0pt}\bfseries\sf#1}}

\newtheorem{thm}{Theorem}[section]
\newtheorem{lem}[thm]{Lemma}
\newtheorem{prop}[thm]{Proposition}
\newtheorem{cor}[thm]{Corollary}
\newtheorem{claim}[thm]{Claim}
\newtheorem{ntt}[thm]{Notation}
\newtheorem{rem}[thm]{Remark}
\newtheorem{prop-def}[thm]{Proposition-Definition}
\newtheorem{ex}[thm]{Example}

\newtheorem*{thmI}{Theorem}
\newtheorem{thmA}{Theorem}[section]

\theoremstyle{definition}
\newtheorem{deff}[thm]{Definition}

\newcommand{\B}{\mathcal{B}}
\newcommand{\D}{\mathcal{D}}
\newcommand{\CC}{\mathbb{C}}
\newcommand{\E}{\mathcal{E}}
\newcommand{\F}{\mathcal{F}}
\newcommand{\I}{\mathcal{I}}

\renewcommand{\L}{\mathcal{L}}
\renewcommand{\O}{\mathcal{O}}
\renewcommand{\P}{\mathcal{P}}
\newcommand{\PP}{\mathbb{P}}
\newcommand{\T}{\mathcal{T}}
\newcommand{\Y}{\mathcal{Y}}

\newcommand{\st}{\;\vline\;} % 'such that' set notation bar
\newcommand{\res}[2]{\left.#1\right|_{#2}} % restriction bar

\newcommand{\abs}[1]{\left\vert#1\right\vert}
\newcommand{\set}[1]{\left\{#1\right\}}

\DeclareMathOperator{\id}{id}
\DeclareMathOperator{\Alb}{Alb}
\DeclareMathOperator{\alb}{alb}
\DeclareMathOperator{\rk}{rk}
\DeclareMathOperator{\cExt}{\mathcal{E}\mathit{xt}}
\DeclareMathOperator{\gv}{gv}
\DeclareMathOperator{\Pic}{Pic}
\DeclareMathOperator{\ch}{char}
\DeclareMathOperator{\supp}{supp}
\DeclareMathOperator{\Bs}{Bs}
\DeclareMathOperator{\codim}{codim}

\newenvironment{enumerate*}{\begin{enumerate}[topsep=4pt, partopsep=4pt,
itemsep=0pt]}{\end{enumerate}}
\newenvironment{itemize*}{\begin{itemize}[topsep=4pt, partopsep=4pt, itemsep=0pt]}{\end{itemize}}
\newenvironment{description*}{\begin{description}[topsep=4pt, partopsep=4pt,
itemsep=0pt]}{\end{description}}

\begin{document}

\title{Generic Vanishing Index and the Birationality of the Bicanonical Map of Irregular Varieties}

\author[M. Lahoz]{Mart\'{\i} Lahoz}
\address{Departament d'\`Algebra i Geometria, Facultat de Matem\`atiques, 
Universitat de Barcelona. Gran Via, 585, 08007 Barcelona. Spain}
\email{{\tt marti.lahoz@ub.edu}}
\curraddr{Mathematisches Institut, Universit\"at Bonn, Endenicher Allee 60, 53115 Bonn, Germany.}

\thanks{The author has been partially supported
by the Proyecto de Investigaci\'on MTM2009-14163-C02-01. This
paper was revised while the author was supported by the SFB/TR 45 `Periods, Moduli Spaces and
Arithmetic of Algebraic Varieties'.}

\begin{abstract} We prove that any smooth complex projective variety with generic vanishing index
bigger or equal than $2$ has birational bicanonical map.
Therefore, if $X$ is a smooth complex projective variety $X$ with maximal Albanese dimension and
non-birational bicanonical map, then the Albanese image of $X$ is fibred by subvarieties of
codimension at most $1$ of an abelian subvariety of $\Alb X$.
\end{abstract}

\maketitle

\markboth{M. Lahoz}{\bf Generic Vanishing Index and the Birationality of the Bicanonical Map}

\section{Introduction}
In the study of smooth complex algebraic varieties, the natural maps provided by the holomorphic
forms defined in the variety, have a special importance. For example, the invertible sheaf
$\omega_X$ of differential $n$-forms (where $n$ is the dimension of $X$) produces a map to a
projective space, known as the canonical map. The multiples of this canonical sheaf
$\omega_X^{\otimes m}$ produce in this way the pluricanonical maps. 
\begin{equation*}
 \varphi_m:X\dashrightarrow \PP^N= \PP(H^0(X,\omega_X^{\otimes m})^\vee).
\end{equation*}
When $\varphi_m$ gives a birational equivalence between $X$ and its image, we will simply say that
$\varphi_m$ is birational.
We say that $X$ is of general type if for some $m>0$ the rational map $\varphi_m$ is birational.

For example, the curves of general type are those of genus $g\geq 2$. The tricanonical map
$\varphi_3$ is always birational for such curves and the bicanonical $\varphi_2$ is also birational
once that $g\geq 3$. Moreover, the canonical map is birational as soon as the curve is
non-hyperelliptic.

For surfaces, Bombieri {\cite{Bom}} have given sharp numerical conditions for the birationality of
$\varphi_m$ for $m\geq 3$. The bicanonical map has revealed to be more complicated and has studied
by many algebraic geometers. 
In fact, the surfaces with irregularity $q(S)\leq 1$ and $\chi(S,\omega_S)=1$ are not completely
understood and there is no classification about which ones have birational $\varphi_2$. For a modern
review of the state of the art in the surface case, we refer to \cite[Thm. 8]{BCP}.

For higher dimensions not many results are known in general. Nevertheless, the example of the
bicanonical map on surfaces shows that for small irregularity $q(X)=h^0(X,\Omega_X^1)$, the
classification becomes more difficult. For complex varieties, recall that the differential $1$-forms
give rise to the Albanese map
\begin{equation*}
  \alb: X\to \Alb X= H^0(X,\Omega_X^1)^\vee/H_1(X,\mathbb{Z}).
\end{equation*}
from $X$ to an abelian variety of dimension $q(X)=h^0(X,\Omega_X^1)$. We say that 
$X$ is \emph{irregular} if, and only if, $\Alb X$ is not trivial, i.e. $q(X)>0$. And we say that $X$
is of \emph{maximal Albanese dimension (m.A.d)} if, and only if, the Albanese map $\alb: X\to \Alb
X$ is generically finite onto its image.

It turns out that some properties of m.A.d varieties seem to behave independently of the dimension
and, indeed, Chen-Hacon showed that this is the case for their pluricanonical maps. 

\begin{thmI}[Chen-Hacon. {\cite{CHlinseries}}]$\;$
\begin{enumerate}
\item $X$ m.A.d and $\chi(\omega_X)>0$ $\Rightarrow$ $X$ is of general type, furthermore, $\varphi_3$
is birational.
\item $X$ m.A.d $\Rightarrow$ $\varphi_6$ is the stable pluricanonical map.
\end{enumerate}
\end{thmI}

For $\varphi_2$, we cannot expect to use $\chi(\omega_X)$ to control directly its birationality. For
example, if $C$ is a curve of genus $2$, then the bicanonical map of the product $C\times Y$ is
never birational. In fact, it is clear that any variety that admits a fibration whose general fibre
has non-birational $\varphi_2$ will have a non-birational bicanonical map. This should be
considered, at least at first glance, as the standard case for higher dimensional varieties. 

The following theorem provides geometric constraints for the non-birationality of the bicanonical
map (see Theorem \ref{thm:corfibr}). 
\begin{thmA}\label{thm:geomfibrA} Let $X$ be a smooth projective complex variety of maximal Albanese
dimension such that the bicanonical map is not birational. Then, the Albanese image of $X$ is
fibred by subvarieties of codimension at most $1$ of an abelian subvariety of $\Alb X$. The base of the fibration is also of maximal Albanese dimension.
\end{thmA}

That is, $X$ admits a fibration onto a normal projective variety $Y$ with $0 \leq \dim Y <
\dim X$, such that any smooth model $\tilde Y$ of $Y$ is of maximal Albanese dimension and
\begin{equation*}q(X) - \dim X \leq q(\tilde Y) - \dim Y +1.
\end{equation*}
Hence, when $q(X)>\dim X+1$ implies the existence of an actual fibration, i.e.~$\dim Y>0$, whose
general fibre is mapped generically finite through the Albanese map of $X$ either onto a fixed
abelian subvariety of $\Alb X$, or onto a divisor of this fixed abelian subvariety. When $\dim Y=0$
the theorem simply says that the image of $X$ in $\Alb X$ has codimension at most $1$.

In particular, when $X$ does not admit any fibration and $q(X)>\dim X$, there is only one possible
case, i.e. $X$ is birationally equivalent to a theta-divisor of an indecomposable principally
polarized abelian variety (see {\cite[Thm. A]{BLNP}}). When $X$ does not admit any fibration and
$q(X)=\dim X$, there is only one known case of variety of general type and non-birational
bicanonical map: a double cover of a principally polarized abelian variety $(A,\Theta)$ branched
along a reduced divisor $B\in \abs{2\Theta}$. Is this the only case? The answer is affirmative in
the case of surfaces due to Ciliberto-Mendes~Lopes \cite[Thm 1.1]{CM}.

To deduce Theorem \ref{thm:geomfibrA} it is useful to consider the generic vanishing index introduced by
Pareschi--Popa in \cite[Def. 3.1]{PPCdF} 
\begin{equation*}
\gv(\omega_X)=\min_{i>0}\set{\codim_{\Pic^0X} V^i(\omega_X)-i},
\end{equation*}
where $V^i(\omega_X)=\set{\alpha\in\Pic^0X\st h^i(X,\omega_X\otimes\alpha)>0}$. As a consequence of
Generic Vanishing Theorem of Green--Lazarsfeld \cite[Thm. 1]{GL1}, we have that for any irregular
variety $1-\dim X\leq \gv(\omega_X)\leq q(X)-\dim X$.

Moreover, the negative values of $\gv(\omega_X)$ can be interpreted in terms of the dimension of the
generic fibre of the Albanese map (see Theorem \ref{thm:GL1}) and $X$ is a m.A.d variety if, and
only if, $\gv(\omega_X)\geq 0$. Due to the work of Pareschi--Popa \cite{PPCdF} we can interpret the
positive values of $\gv(\omega_X)$ in terms of the local properties of the Fourier-Mukai transform
of the structural sheaf (see Theorem \ref{thm:posit_gv}). They have also proved that the positive
values of $\gv(\omega_X)$ give a lower bound for the Euler characteristic $\chi(\omega_X)$ (see
Theorem \ref{thm:CdF}).

Using the generic vanishing index we have the following more synthetic result. 
\begin{thmA} Let $X$ be a smooth projective complex variety such that $\gv(\omega_X)\geq 2$. Then,
the rational map associated to $\omega_X^{2}\otimes \alpha$ is birational onto its image for every
$\alpha\in\Pic^0 X$.
\end{thmA}

Theorem \ref{thm:geomfibrA} is deduced from this result by an argument of Pareschi-Popa. On the
other hand, this result (see Theorem \ref{thm:fibr}) is proved using a birationality criterion (see
Lemma \ref{lem:birprov}) that is a slight modification of \cite[Thm. 4.13]{BLNP}.

For curves, $\gv(\omega_C)\geq 2$ is equivalent to $g(C) \geq 3$. For surfaces, $\gv(\omega_S)\geq 2$ is equivalent to suppose that $q(S)\geq 4$ and does not admit an irregular fibration to a curve of genus $\leq q(S)-3$ (see Example \ref{ex:exemples}). \\ 
 
\noindent{\bf Acknowledgments.} This is part of my Ph.D. thesis. I would like to thank my advisors,
M.A.~Barja and J.C.~Naranjo, for their valuable suggestions. I am also grateful to G.~Pareschi for
helpful conversations during my Ph.D studies. I would like to thank 
also R.~Pardini for her useful comments.

\section{Generalized Fourier-Mukai transform}
$X$ will be a smooth projective variety over an algebraically closed field $k$ (from section
\ref{sub:char0} on, we will restrict to $k=\CC$). It will be equipped with a morphism $a:X\to A$ to
a non-trivial abelian variety $A$, in particular, $X$ will be irregular. Let $\P$ be a Poincar\'e
line bundle on $A\times \Pic^0A$. We will denote
\begin{equation}\label{eq:PoicPa} 
P_a=(a\times \id_{\Pic^0X})^*\P,
\end{equation}
the \emph{induced Poincar\'e line bundle} in $X\times \Pic^0A$. When $a=\alb$, the Albanese map of
$X$,  then the map $\alb^*$ identifies $\Pic^0(\Alb X)$ to $\Pic^0X$ and the line bundle $P_{\alb}$
will be simply denoted by $P$.

Letting $p$ and $q$ the two projections of $X\times \Pic^0 A$, we consider the left exact functor
$\Phi_{P_a} \F=q_*(p^*\F\otimes P_a)$, and its right derived functors 
\begin{equation}\label{eq:FMPa} R^i\Phi_{P_a} \F= R^iq_*(p^*\F\otimes P_a).
\end{equation}
Sometimes we will have to consider  the analogous derived functor 
$R^i\Phi_{P_a^{-1}}\F$ as well. By the Seesaw Theorem \cite[Cor. 6, pg. 54]{MAV}, $\P^{-1}=
(1_A\times
(-1)_{\Pic^0 A})^*\P$, so 
\begin{equation}\label{eq:FMPdual} R^i\Phi_{P_a^{-1}}\F= (-1_{\Pic^0 A})^*R^i\Phi_{P_a}\F
\qquad\text{for any }i.
\end{equation}
Given a coherent sheaf $\F$ on $X$, its \emph{i-th cohomological support locus  with respect to $a$}
is
\begin{equation*}V^i_a(\F)=\set{\alpha \in \Pic^0A \st h^i(\F\otimes a^*\alpha)>0}\end{equation*}
Again, when $a$ is the Albanese map of $X$, we will omit the subscript, simply writing
$V^i(\F)$. 
By base change, these loci contain the set-theoretical support of $R^i\Phi_{P_a}\F$, i.e.
$\supp R^i\Phi_{P_a}\F\subseteq V^i_a(\F)$.

A way to measure the size of all the $V^i_a(\F)$'s is provided by the following invariant introduced
by Pareschi--Popa.
\begin{deff}[{\cite[Def. 3.1]{PPCdF}}]\label{def:gvindex} Given a coherent sheaf $\F$ on $X$, the
\emph{generic vanishing index} of $\F$ (with respect to $a$) is
\begin{equation*}\gv_a(\F):=\min_{i>0}\set{\codim_{\Pic^0 A}V^i_a(\F)-i}.\end{equation*}
By convention we define $\gv_a(\F)=\infty$, when $V^i_a(\F)=\emptyset$ for every $i>0$. When $a$ is
the Albanese map of $X$, we will omit the subscript, simply writing 
$\gv(\F)$.
\end{deff}
By base change (see \cite[Lem. 2.1]{PPCdF}) it is easy to see that $\gv_a(\F)$ can be also defined
as the  $\min_{i>0}\set{\codim_{\Pic^0 A}\supp R^i\Phi_{P_a}\F - i}$.

\section{Generic vanishing index of the canonical sheaf}

\subsection{Relations between $\gv(\omega_X)$ and the Fourier-Mukai transform of $\O_X$}
Here we specialize some general results of Pareschi--Popa \cite{PPCdF,PPGVsheaves} to the canonical
sheaf of a smooth projective variety  of dimension $d$. Some of these results were previously
obtained by Hacon (see \cite{Happroach}).

The negative values of the $\gv$-index are related with the vanishing of the lowest cohomologies of the
Fourier-Mukai transform of its Grothendieck dual. In the case of $\omega_X$ this can be stressed
simply as:
\begin{thm}[{\cite[Thm. 2.2]{PPCdF}}]\label{thm:GVvsWIT} 
The following are equivalent,
\begin{enumerate*}
 \item $\gv_a(\omega_X)\geq -e$ for $e\geq 0$;
 \item $R^i\Phi_{P_a}\O_X =0$ for all $i\neq d-e,\ldots,d$.
\end{enumerate*}
\end{thm} 
 Hence, when $\gv_a(\omega_X)\geq 0$, $R^i\Phi_{P_a}\O_X=0$ for all $i\neq d$, and we usually denote
\begin{equation*}\widehat{\O_X}=R^d\Phi_{P_a} \O_X.\end{equation*}
Note that, in this case, $H^i(X,\omega_X\otimes a^*\alpha)=0$ for all $i>0$ and general $\alpha\in
\Pic^0 A$. Therefore, by deformation-invariance of $\chi$, the generic value of
$h^0(X,\omega_X\otimes a^*\alpha)$ equals $\chi(\omega_X)$, in particular $\chi(\omega_X)\geq 0$.
Since, by base-change, the fibre of $\widehat{\O_X}$ at a general point $\alpha\in \Pic^0 A$ is
isomorphic to $H^{d}(X,a^*\alpha)\cong H^0(X,\omega_X \otimes a^*\alpha^{-1})^*$, the (generic) rank
of $\widehat{\O_X}$ is $\rk \widehat{\O_X}=\chi(\omega_X)$.

From Grothendieck-Verdier duality \cite[Thm. 4.3.1]{conrad} and Theorem \ref{thm:GVvsWIT} it follows
that,
\begin{cor}[{\cite[Rem. 3.13]{PPGVsheaves}}]\label{cor:duality}  If $\gv_a(\omega_X)\geq 0$  then
$\cExt^i_{\O_{\Pic^0 A}}((-1_{\Pic^0 A})^*\widehat{\O_X},\O_{\Pic^0 A})\cong 
R^i\Phi_{P_a}\omega_X$. 
\end{cor}
The following result of Pareschi--Popa gives a dictionary between the positive values of
$\gv_a(\omega_X)$ and the local properties of the Fourier-Mukai transform of $\widehat{\O_X}$. 
\begin{thm}[{\cite[Cor. 3.2]{PPCdF}}] \label{thm:posit_gv}
Assume that $\gv_a(\omega_X)\geq 0$. Then, 
\begin{equation}
\gv_a(\omega_X)\geq m\;\text{ if, and only if, }\;\widehat{\O_X}\text{ is a }m\text{-syzygy sheaf.}
\end{equation}
In particular, $\gv_a(\omega_X)\geq 1$ is equivalent to $\widehat{\O_X}$ being torsion-free and
$\gv_a(\omega_X)\geq 2$ to $\widehat{\O_X}$ being reflexive.
\end{thm}

Using the Evans--Griffith Syzygy Theorem and the previous theorem, Pareschi--Popa obtain the
following bound on the Euler holomorphic characteristic that generalizes to higher dimensions the
Castelnuovo-de~Franchis inequality.
\begin{thm}[{\cite[Thm. 3.3]{PPCdF}}]\label{thm:CdF} Assume that $\gv_a(\omega_X)\geq 0$. Then, 
$\chi(\omega_X) \geq \gv_a(\omega_X)$. 
\end{thm}
\begin{rem}\label{rem:CdF} In fact, the theorem of Pareschi--Popa is more general, namely that for
any coherent sheaf $\F$ if $\infty>\gv_a(\F)\geq 0$, then $\chi(\F)\geq \gv_a(\F)$. As a
consequence, we easily obtain that for any non-zero coherent sheaf $\F$, $\gv_a(\F)\geq 1\Rightarrow
\chi(\F)\geq 1$. Observe also that if $a$ is non-trivial, we always have $\gv_a(\omega_X)<\infty$.
\end{rem}

\subsection{Top Fourier-Mukai transform of the canonical sheaf}
In the case of abelian varieties (or complex torus) the following result is well-known and crucial
in the proof of the Mukai Equivalence Theorem \cite[Thm 2.2]{Mdual}. We will need it in the proof of
Theorem \ref{thm:fibr}.
\begin{prop}[{\cite[Prop. 6.1]{BLNP}}]\label{prop:IrrMukai} If $a^*:\Pic^0A\to \Pic^0X$ is an
embedding, then \begin{equation*}R^d\Phi_{P_a} \omega_X\cong k(\hat 0). 
\end{equation*}
\end{prop}

\subsection{Generic vanishing theorem of Green--Lazarsfeld}\label{sub:char0}
The name of the $\gv$-index comes from the well-known Generic Vanishing Theorem of
Green--Lazarsfeld. As other general vanishing theorems, it requires $\ch k=0$ so from now on we will
restrict ourselves to the case $k=\CC$.
Basically, the following theorem is \cite[Thm. 1]{GL1}. The converse implication was proven
independently in \cite[Thm. B]{LP} and \cite[Prop. 2.7]{BLNP}. 
\begin{thm}\label{thm:GL1} For any $e>0$, the following are equivalent:
\begin{enumerate*}
\item the generic fibre of $a\colon X\to A$ has dimension $e$,
\item $\gv_a(\omega_X)= -e$.
\end{enumerate*}
Moreover $\gv_a(\omega_X)\geq 0$ if, and only if, $a\colon X\to A$ is generically finite onto its
image. 
\end{thm}
In particular, observe that for any irregular variety
$1-\dim X\leq \gv(\omega_X)\leq q(X)-\dim X$.

\begin{rem}\label{rem:gt}
If $\gv_a(\omega_X)\geq 0$ and $\chi(\omega_X)>0$, then $X$ is a variety of general type. Indeed, by
the previous result $a:X\to A$ is generically finite and since $\chi(\omega_X)>0$, we have that
$V_a^0(\omega_X)=\Pic^0A$, so by \cite[Cor.2.4]{CHpm}, $\kappa(X)=\dim X$. In particular, if
$\gv_a(\omega_X)\geq 1$, then $X$ is of general type.
\end{rem}

\subsection{Subtorus theorem of Green--Lazarsfeld and Simpson}
The following theorem is due to Green and Lazarsfeld \cite[Thm 0.1]{GL2} with an important addition
due to Simpson \cite[\S4,6,7]{simpson}. 
\begin{thm}\label{thm:GL2} Let $W$ an irreducible component
of $V^i(\omega_X)$ for some $i$. Then,
\begin{enumerate*}
 \item There exists a torsion point $\beta\in \Pic^0X$ and a subtorus $B$ of $\Pic^0X$ such that
$W=\beta +B$.
 \item There exists a normal variety $Y$ of dimension $\leq d-i$, such that any smooth model of $Y$
has maximal Albanese dimension and a morphism with connected fibres $f\colon X\to Y$ such that $B$
is contained in $f^*\Pic^0Y$.
\end{enumerate*}
\end{thm}

\begin{rem}\label{rem:Steinfact}
It is useful to recall that the morphism $f\colon X\to Y$ in the second part of the previous
theorem, arises as the Stein factorization of the morphism $\pi\circ \alb\colon X\to \Pic^0W$, where
$\pi\colon\Alb X\to \Pic^0W$ is the dual map of the inclusion $W\subseteq \Pic^0X$. Hence, the key
point of the second part of the theorem is the dimensional bound for $Y$.
\end{rem}

\section{Birationality criterion for maximal Albanese dimension varieties}\label{sec:bir2}
In this section, we will assume that $a:X\to A$ is a generically finite morphism onto its image,
where $A$ is an abelian variety.
We introduce another piece of notation. 
\begin{ntt}\label{term:last} Let $\F$ be a subsheaf of a line bundle  and suppose that
$\gv_a(\F)\geq 1$.
\begin{enumerate*} 
 \item We denote $U_\F$, the open subset where $h^0(\F\otimes a^*\alpha)$ has the minimal value,
i.e. $\chi(\F)$.
 \item  Let $Z$ be the exceptional locus of $a:X\to A$, that is $Z = a^{-1} (T)$, where $T$ is
the locus of points in $A$ over which the fibre of $a$ has positive dimension. 
 \item We define
\begin{equation*}
\B^\F_a(x)=\set{\alpha\in U_\F\st x\text{ is a base point of } \abs{\F\otimes a^*\alpha}}.
\end{equation*}
By Remark \ref{rem:CdF}, $\chi(\F)\geq 1$. So, by semicontinuity, it makes sense to speak of the
base locus of $\F\otimes a^*\alpha$ for all $\alpha\in\Pic^0A$.
\end{enumerate*}
\end{ntt}

The following lemma is a slight modification of \cite[Thm. 4.13]{BLNP} and it is based on \cite[Prop. 2.12 and 2.13]{PPreg1}.
\begin{lem}\label{lem:birprov} 
Suppose that $a:X\to A$ is a generically finite morphism onto its image and
let $\F$ be a subsheaf of a line bundle  such that $\gv_a(\F)\geq 1$ and $R^ia_*\F=0$ for all
$i>0$. Suppose that for a general $x\in X$,
\begin{equation*}
 \codim_{U_\F} \B^{\F}_a(x)\geq 2.
\end{equation*}
Then,  the rational map associated to the linear system $\abs{\F\otimes L}$
is birational for every line bundle $L$ such that  $\gv_a(L)\geq 1$.
\end{lem} 

\begin{proof}[Proof] 
We first compare the Fourier-Mukai transform of $\F\otimes \I_x$ and $\F$. 

\noindent\textbf{Claim.} Let $x\in X$ be a closed point out of $Z$. Then
$R^ia_*(\F\otimes \I_x \otimes a^*\alpha)=0$ for $i>0$.
This follows immediately from the exact sequence 
 \begin{equation}\label{eq:standard}
0\to \F\otimes \I_x \to \F\to k(x)\to 0
\end{equation} 
and the hypothesis that $R^ia_*\F=0$, $a$ is generically finite and $x\not\in Z$. Hence, the
degeneration of the Leray spectral sequence yields to
\begin{equation}\label{eq:GRp}V^i_a(\F\otimes \I_x)=V^i(a_*(\F\otimes \I_x)).
\end{equation}
By sequence \eqref{eq:standard}, tensored by $a^*\alpha$, it follows that
\begin{equation}\label{eq:equality}V^i_a(\F\otimes \I_x)=V^i_a(\F)\qquad\text{for all } i\geq 2.
\end{equation}
For $i=1$ we have the surjection $H^1(\F\otimes \I_x\otimes a^*\alpha)\twoheadrightarrow
H^1(\F\otimes a^*\alpha)$, 
that is an isomorphism if, and only if, $x$ is not a base point of $\abs{\F\otimes a^*\alpha}$.
In other words
$V^1_a(\F\otimes \I_x)\subseteq \B^\F_a(x)\cup V^1_a(\F)$.
Since $\gv_a(\F)\geq 1$, the hypothesis on $\B^{\F}_a(x)$ guarantees that
\begin{equation}\label{eq:codim} \codim V^1_a(\F\otimes \I_x)\geq 2,
\end{equation}
for a general $x\in X\setminus Z$. Hence by \eqref{eq:GRp}, \eqref{eq:equality} and
\eqref{eq:codim}, $\gv(a_*(\F\otimes \I_x))\geq 1$. By \cite[Prop. 2.13]{PPreg1}, 
$a_*(\F\otimes \I_x)$ is continuously globally generated (CGG, see \cite{PPreg1}).
Therefore $\F\otimes \I_x$ itself is CGG outside $Z$ (with respect to $a$). Since the same is true
for $L$, it follows from \cite[Prop 2.12]{PPreg1} that for all $\alpha\in\Pic^0 A$, $\F\otimes
L\otimes\I_x$ is globally generated outside $Z$. So the rational map associated to
$\abs{\F\otimes L}$ is birational.
\end{proof}

\begin{rem}\label{rem:birvsva}
From the proof we see that if $\codim_{U_\F} \B^{\F}_a(x)\geq 2$ for every $x\in X\setminus Z$, then
$\F\otimes L$ is very ample out of $Z$, the exceptional locus of $a$. 
\end{rem}

\subsection{Adjoint line bundles} \label{subsect:adj}
When $\F=\omega_X$ we will call $U_\F$ simply $U_0$ and $\B^{\omega_X}_a(x)$ simply by 
\begin{equation}
\B_a(x)=\set{\alpha\in U_0\st x\text{ is a base point of }\omega_X \otimes a^*\alpha }. 
\end{equation} 
Throughout subsections \S\ref{subsect:adj} and \S\ref{subsect:decomp}, we will assume that
$\gv_a(\omega_X)\geq 1$.

\begin{prop-def}\label{pr-df:birprov} 
Let $X$ be a variety such that $\gv_a(\omega_X)\geq 1$ and let $L$ be any line bundle on $X$ such
that $\gv_a(L)\geq 1$. Suppose that there exists $\alpha\in \Pic^0A$ such that $\omega_X\otimes
L\otimes a^*\alpha$ is not birational.
Then, 
\begin{equation*}
 \codim_{X\times U_0} \set{(x,\alpha)\in X\times U_0\st x\text{ is a base point of }\omega_X \otimes
a^*\alpha }=1,
\end{equation*}
and its divisorial part is dominant on $X$ and surjects on $U_0$ via the projections $p$ and $q$. We
endow this set with the  natural subscheme structure given by the image of the relative evaluation
map $q^*(q_*\L)\otimes\L^{-1}\to \O_{X\times U_0}$, where $\L= \res{\left(p^*\omega_X)\otimes
P_a\right)}{X\times U_0}$ and we call $\Y$ the union of its divisorial components that dominate
$U_0$. Let $\overline{\Y}$ be its closure in $X\times \Pic^0 X$. Then 
\begin{enumerate*}
 \item  $X$ is covered by the scheme-theoretic fibres of the projection $\overline{\Y}\to U_0$, that
we will call 
$F_\alpha$, for $\alpha$ varying in $U_0$. By definition, at a \emph{general} point $\alpha\in U_0$,
$F_\alpha$ is the fixed divisor of $\omega_X\otimes a^*\alpha$.
 \item For a general $x\in X$, the fibre of the projection $\overline\Y\to X$ is a divisor, that we
will call $\D_x$. By definition,
$\D_x$ is the closure of the union of the divisorial components of the locus of $\alpha\in U_0$ such
that $x\in \Bs(\omega_X\otimes a^*\alpha)$. 
\end{enumerate*}
\end{prop-def}

\begin{proof}
Everything follows from taking $\F=\omega_X$ in Lemma \ref{lem:birprov}. The surjectivity of the
projection to $U_0$ is consequence of the  Castelnuovo-de~Franchis inequality \ref{thm:CdF},
i.e.~$\chi(\omega_X)\geq \gv_a(\omega_X)\geq 1$.
\end{proof}

\subsection{Decomposition} \label{subsect:decomp} In the sequel we will need $a^*:\Pic^0A\to
\Pic^0X$ to be an embedding. However, for simplicity we will go one step further and we will simply
suppose that $A=\Alb X$. 
Suppose that we are under the hypotheses of the previous Proposition-Definition and consider a fixed
point $\alpha_0\in U_0$, and the map
\begin{equation}\label{eq:f}f_{\alpha_0}\colon U_0 \to \Pic^0 X\qquad \alpha\mapsto
\O_X(F_\alpha-F_{\alpha_0}),
\end{equation}
where $F_\alpha$ is the divisor defined in Proposition-Definition \ref{pr-df:birprov}{\it(a)}. For $\alpha \in U_0$, all the $F_\alpha$ are algebraically equivalent since they are the fibres of $\overline{\Y}\to U_0$, so the map is well-defined.

The following lemma shows that this map induces a decomposition of $\Pic^0X$ and that the divisors
$F_\alpha$ move algebraically along a non-trivial factor of $\Pic^0X$. Although the proof is
basically the same as \cite[Lem. 5.1]{BLNP}, we do not require $V^1(\omega_X)$ to be a finite set,
but only a proper subvariety.
\begin{lem}\label{lem:propKerf} The map defined in \eqref{eq:f}, induces an homomorphism $f:\Pic^0
X\to\Pic^0X$ such that,
\begin{enumerate*}
 \item $f^2=f$ and $\Pic^0 X$ decomposes as $\Pic^0 X\cong \ker f\times \ker(\id- f).$ Moreover
$\dim \ker(\id-f)>0$.
 \item Fix $\bar\beta\in \ker f$ such that $ U_0\cap\left(\set{\bar\beta}\times\ker(\id-f)\right)$ 
is non-empty. Then, for $\gamma\in U_0\cap \ker({\rm id}-f)$ the line bundle
$\O_X(F_{\bar\beta\otimes \gamma})\otimes \gamma^{-1}$ does not depend on $\gamma$. Since it is
effective by semicontinuity, we call it $\O_X(F)$. 
\item For all $(\beta,\gamma)\in \ker f\times\ker({\rm id}-f)\cong\Pic^0 X$ such that
$\beta\otimes\gamma\in U_0$, $\abs{\O_X(F)\otimes\gamma}$ is contained in the fixed divisor of
$\omega_X\otimes\beta\otimes\gamma$.
\end{enumerate*}
\end{lem}
\begin{proof} Let $\O_X(M_\alpha)=\omega_X\otimes a^*\alpha \otimes \O_X(-F_\alpha)$. Then, the
proof of {\it(a)} is the same as \cite[Lem. 5.1]{BLNP}{\it(a)}. Item {\it(b)} follows directly from
the definition of $f$. To prove {\it(c)}, let  $(\beta,\gamma)\in \ker f\times\ker(\id-f) $ such
that $\beta\otimes\gamma\in U_0$ and $E\in \abs{\O_X(F)\otimes\gamma}$. Then
$\O_X(F_{\beta\otimes\gamma}-E)\cong\O_X(F_{\beta\otimes\gamma}-F_{\bar\beta\otimes\gamma}
)=f(\beta\otimes\bar\beta^{-1})=\O_X$. Since $F_{\beta\otimes\gamma}$ is a fixed divisor of
$\abs{\omega_X\otimes\beta\otimes\gamma}$, also $E=F_{\bar\beta\otimes\gamma}$ is a fixed divisor in
$\abs{\omega_X\otimes\beta\otimes\gamma}$. \qedhere
\end{proof}
Using the decomposition given by the previous Lemma we give an explicit description of the ``half''
Poincar\'e line bundle.
\begin{lem}[{\cite[Lem. 5.1, 5.3]{BLNP}}]\label{lem:halfpoinc}\label{lem:decomp}
We call $B=\Pic^0(\ker f)$ and $C=\Pic^0(\ker (\id- f))$ so  that
\begin{equation*}
\Alb X\cong B\times C\qquad \text{ and }\qquad\Pic^0 X\cong \Pic^0 B\times \Pic^0C,
\end{equation*}
with $\dim C>0$. Then we have the following description of ``half'' Poincar\'e line bundle.
\begin{equation*}(\alb\times \id_{\Pic^0 X})^*(\O_{B\times\Pic^0 B}\boxtimes\P_C)\cong
\O_{X\times\Pic^0 X}(\overline\Y)\otimes p^*\O_X(-F)\otimes q^*\O_{\Pic^0 X}(-\D_{\bar x}),
\end{equation*}
where $\bar x$ is such that $\alb(\bar x)=0$ in $\Alb X$ and $\P_C$ is the Poincar\'e line bundle in
$C\times \Pic^0C$.
\end{lem}
\begin{proof} The decomposition of $\Pic^0X$ comes directly from Lemma \ref{lem:propKerf}{\it(a)}. 
By the definition of $\overline\Y$ (see Proposition-Definition \ref{pr-df:birprov}) and the
definition of $F$ (see Lemma \ref{lem:propKerf}{\it(b)}) we have that the line bundle
\begin{equation*}\O_{X\times\Pic^0 X}(\overline\Y)\otimes p^*\O_X(-F)\otimes q^*\O_{\Pic^0
X}(-\D_{\bar p}),\end{equation*} 
\begin{itemize*}
\item[-] restricted to $X\times\set{\beta\otimes\gamma}$ is isomorphic to
$\O_X(F_{\beta\otimes\gamma}-F)=\gamma$,
for all $(\beta,\gamma)\in U_0\subseteq \ker f\times\ker(\id -f)$;
\item[-] restricted to $\set{\bar x}\times\Pic^0 X$ is isomorphic to $\O_{\Pic^0X}(\D_{\bar
x})\otimes\O_{\Pic^0X}(-\D_{\bar x})$, i.e. trivial.
\end{itemize*}
On the other hand, $(\alb\times \id_{\Pic^0 X})^*(\O_{B\times\Pic^0 B}\boxtimes\P_C)$,
\begin{itemize*}
\item[-] restricted to $X\times\set{\beta\otimes\gamma}$ is isomorphic to $\gamma$, for all
$(\beta,\gamma)\in\ker f\times\ker(\id -f)$; 
\item[-] restricted to $\set{\bar x}\times\Pic^0 X$ is isomorphic to  $\O_{\Pic^0X}$, i.e. trivial.
\end{itemize*}
Then, the Lemma follows from the see-saw principle.
\end{proof}

\section{The bicanonical map of irregular varieties}\label{sec:bic}
The next theorem gives a sufficient numerical condition for the birationality of the bicanonical
map, analogous to Pareschi--Popa Theorem {\cite[Thm. 6.1]{PPreg3}} %\ref{cor:rep} 
for the tricanonical map.
\begin{thm}\label{thm:fibr} Let $X$ be a smooth projective complex variety such that
$\gv(\omega_X)\geq 2$. Then, the rational map associated to $\omega_X^{2}\otimes \alpha$ is
birational onto its image for every $\alpha\in\Pic^0 X$.
\end{thm}
As a first corollary we have the following geometric interpretation.
\begin{thm}\label{thm:corfibr} Let $X$ be a smooth projective complex variety of maximal Albanese
dimension such that the bicanonical map is not birational. Then $0 \leq \gv(\omega_X)\leq 1$.
Moreover, it admits a fibration onto a normal projective variety $Y$ with $0 \leq \dim Y < \dim X$, 
any smooth model $\tilde Y$ of $Y$ is of maximal Albanese dimension and 
\begin{equation*}
q(X) - \dim X \leq q(\tilde Y) - \dim Y +\gv(\omega_X).
\end{equation*}
\end{thm}

\begin{proof} By Theorems \ref{thm:GL1} and \ref{thm:fibr}, it is clear that $0 \leq
\gv(\omega_X)\leq 1$. Now, the proof is the same as the proof of \cite[Thm. B]{PPCdF}.
\end{proof}

\begin{ex}\label{ex:exemples} We would like to show examples of varieties with $\gv(\omega_X)\geq 2$. 
For curves $C$, this is equivalent to $g(C) \geq 3$. For surfaces $S$, is equivalent to suppose that $q(S)\geq 4$ and $S$ does not admit an irregular fibration to a curve of genus $\leq q(S)-3$ (see \cite[Cor. 2.3]{Be}).

On the other hand, if $A$ is a simple abelian variety, then every subvariety $X$ of codimension $\geq 2$ has $\gv(\omega_X)\geq 2$.
Moreover, the property of having $\gv(\omega_X)\geq 2$ is closed under taking products and cyclic coverings induced by a torsion point $\alpha\in \Pic^0X-V^1(\omega_X)$. 
\end{ex}

The rest of the paper is devoted to the proof of Theorem \ref{thm:fibr}.
\begin{proof} Assume that $\gv(\omega_X)\geq 1$ and there exists $\alpha\in \Pic^0X$ such that
$\omega_X^{\otimes 2}\otimes \alpha$ is non-birational. Then, we want to see that
$\gv(\omega_X)=1$. 
Under these hypotheses we can apply Proposition-Definition \ref{pr-df:birprov} and Lemma
\ref{lem:halfpoinc}, so $\Alb X\cong B\times C$, where $B=\Pic^0(\ker (\id -f))$ and $C=\Pic^0(\ker
f)$. We have the following commutative diagram 
\begin{equation}\label{eq:2diagrames}
\xymatrix{
\Pic^0X\ar[d]_{p_{\hat{b}}}&X\times \Pic^0X\ar[l]_q \ar[r]^{\alb\times\id}\ar[d]^{\id\times p_{\hat
b}} &\Alb X\times \Pic^0X\ar[d]^{p_b\times p_{\hat{b}}}\\
\Pic^0B& X\times\Pic^0B\ar[l]^q \ar[r]_{b\times \id} & B\times\Pic^0B}
\end{equation}
where 
\begin{itemize*}
 \item[-] $p_b:\Alb X\to B$ and $p_{\hat b}:\Pic^0X\to\Pic^0B$ are the corresponding projections,
 \item[-] $b$ is the composition by $b:X\stackrel{\alb}{\to}\Alb X \stackrel{p_b}{\to} B$, and
 \item[-] abusing notation we also call $q$ either the projection $X\times\Pic^0X\to \Pic^0X$ or
$X\times \Pic^0B\to\Pic^0B$ and $p$ the projections $X\times\Pic^0X\to X$ or $X\times \Pic^0B\to
X$. 
\end{itemize*}
The effectiveness of $\overline\Y$ give us the following short exact sequence on $X\times \Pic^0 X$
\begin{equation*}
0\to (\alb\times \id)^*(\O_{B\times\Pic^0 B}\boxtimes\P_C)^{-1} \stackrel{\cdot \overline\Y}\to
p^*\O_X(F)\otimes q^*\O({\mathcal D}_{\bar x})\to \res{\left(p^*\O_X(F)\otimes q^*\O({\mathcal
D}_{\bar x})\right)}{\overline\Y}\to 0.
\end{equation*}
Recall that $P=(\alb\times \id_{\Pic^0 X})^*(\P_B\boxtimes\P_C)$ since the Poincar\'e line bundle
$\P$ in $\Alb X\times\Pic^0X$ is isomorphic to $\P_B\boxtimes\P_C$. We apply the functor $R^d
q_*(\>\cdot\>\otimes(\alb\times\id_{\Pic^0 X})^*(\P_B^{-1}\boxtimes\O_{C\times\Pic^0 C}))$, that is,
we tensor by the other ``half'' Poincar\'e line bundle and we consider the top direct image. We get
\begin{align*}
\cdots\to &R^d\Phi_{P^{-1}}(\O_X)\to  R^d q_*\left(p^*\O_X(F)\otimes (\alb\times\id_{\Pic^0
X})^*(\P_B^{-1}\boxtimes\O_{C\times\Pic^0 C})\right)\otimes \O_{\Pic^0X}(\mathcal{D}_{\bar x})\to \\
&\to R^d q_*\left(\res{(p^*\O_X(F)\otimes q^*\O_{\Pic^0X}(\D_{\bar x}))}{\overline\Y}\otimes
(\alb\times\id_{\Pic^0 X})^*(\P_B^{-1}\boxtimes\O_{C\times\Pic^0 C})\right)\to 0 
\end{align*}
Using that $R^i\Phi_{P^{-1}}\cong (-1)_{\Pic^0 X}^*R^i\Phi_{P}$ (see \eqref{eq:FMPdual}), we have
the
following short exact sequence,
\begin{equation}\label{eq:fibrshort} 0\to (-1)^*_{\Pic^0 X}\widehat{\O_X}\buildrel\mu\over\to
\E({\mathcal D}_{\bar x})\to \T\to 0\end{equation}
where:
\begin{enumerate*}
\item By base change, $\E= R^d q_*(p^*\O_X(F)\otimes(\alb\times\id_{\Pic^0
X})^*(\P_B^{-1}\boxtimes\O_{C\times\Pic^0 C}))$ is a coherent sheaf of rank
$h^d(\O_X(F)\otimes\beta^{-1})$ by a general $\beta\in \ker f$, i.e
$h^0(\omega_X\otimes\O_X(-F)\otimes\beta)=\chi(\omega_X)$ by Lemma \ref{lem:propKerf}{\it(c)}. Then,
\begin{align*}
\E&= R^d q_*(p^*\O_X(F)\otimes(\alb\times\id)^*(\P_B^{-1}\boxtimes\O_{C\times\Pic^0 C}))\\
&= R^d q_*(p^*\O_X(F)\otimes (\id\times p_{\hat b})^*(b\times\id)^*\P_B^{-1}) &\text{right square of
\eqref{eq:2diagrames}}\\
&= R^d q_*(\id\times p_{\hat b})^* (p^*\O_X(F)\otimes (b\times\id)^*\P_B^{-1}) &\text{abuse of
notation on }p\\
&= p_{\hat b}^*R^d q_*(p^*\O_X(F)\otimes (b\times\id)^*\P_B^{-1}) &\text{flat base change}\\
&= p_{\hat b}^*R^d\Phi_{P_b^{-1}}(\O_X(F)),
\end{align*}
following the notation of \eqref{eq:PoicPa} and \eqref{eq:FMPa}.
\item $\T=R^d q_*\left(\res{(p^*\O_X(F)\otimes q^*\O_{\Pic^0X}(\D_{\bar x}))}{\overline\Y}\otimes
(\alb\times\id_{\Pic^0 X})^*(\P_B^{-1}\boxtimes\O_{C\times\Pic^0 C})\right)$ is supported at the
locus of the $\alpha\in\Pic^0 X$ such that the fibre of the projection $q\colon\overline\Y\to \Pic^0
X$ has dimension $d$, i.e. it coincides with $X$. Such locus is contained in $V^1(\omega_X)$,
therefore, since $\gv(\omega_X)\geq 1$, $\codim \supp \T \geq 2$. 
\item The map $\mu$ is injective since it is a generically surjective map of sheaves of the same
rank (recall that $\rk\widehat{\O_X}=\chi(\omega_X)$), and, as $\gv(\omega_X)\geq 1$, the source
$\widehat{\O_X}$ is torsion-free (Thm. \ref{thm:posit_gv}).
\item $\mu$ is $R^d q_*(m_s)$, where $m_s$ is the multiplication by the section defining
$\overline\Y$. By base change \cite[Cor. 3, pg. 53]{MAV},
$R^d q_*(m_s)\otimes \CC(\alpha)= H^d(\res{m_s}{q^{-1}\set{\alpha}})$ where $q$ is  the projection
$q\colon\overline\Y\to \Pic^0 X$.
When $q^{-1}\set{\alpha}=X$, $\res{m_s}{q^{-1}\set{\alpha}}=0$, so in these points $R^d
q_*(m_s)\otimes \CC(\alpha)=0$.
\end{enumerate*}
\begin{claim} $\T\neq 0$.
\end{claim}
\begin{proof}[Proof of the Claim]
Suppose that $\T=0$, so $\mu$ is an isomorphism. Taking $\cExt^d(\>\cdot\>,\O_{\Pic^0X})$ we get
\begin{align*}
k(\hat 0) &=R^d\Phi_{P}\omega_X&\text{Prop. \ref{prop:IrrMukai}}\\
&=\cExt^d(\E,\O_{\Pic^0X})\otimes\O(-\D_{\bar x}) &\cExt^d(\mu,\O_{\Pic^0X})\text{ and Cor.
\ref{cor:duality}}\\
&=p_{\hat b}^*\cExt^d(R^d\Phi_{P_b}(\O_X(F)),\O_{\Pic^0B})\otimes\O(-\D_{\bar x}) &\text{see item
{\it
(a)},}
\end{align*}
which implies that $\codim_{\Alb X} B=\dim \ker(\id -f)=0$ contradicting Lemma \ref{lem:decomp}.
\end{proof}

Let $\tau(\E(\D_{\bar x}))$ be the torsion part of $\E(\D_{\bar x})$ and $\widetilde{\E(\D_{\bar
x})}$ the quotient of $\E(\D_{\bar x})$ by its torsion part. Hence $\widetilde{\E(\D_{\bar x})}$ is
torsion-free. Now consider the following composition
\begin{equation*}
\xymatrix@C=1.8pc@R=1.2pc{
(-1)^*_{\Pic^0 X}\widehat{\O_X}\ar[r]^(.6)\mu\ar[rd]_{\tilde\mu} & \E({\mathcal D}_{\bar
x})\ar@{>>}[d]\\
&\widetilde{\E(\D_{\bar x})}.}
\end{equation*}
Since $\tilde\mu$ is generically surjective and $(-1)^*_{\Pic^0 X}\widehat{\O_X}$ is torsion-free
(recall that $\gv(\omega_X)\geq 1$), we have that $\tilde\mu$ is injective. Completing the diagram
we get,
\begin{equation}\label{eq:tautilde}
\xymatrix@C=1.3pc@R=1.pc{&& 0\ar[d]&0\ar[d]\\
&& \tau(\E(\D_{\bar x}))\ar[d]\ar@{=}[r]&\tau(\E(\D_{\bar x}))\ar[d]\\
0\ar[r] &(-1)^*_{\Pic^0 X}\widehat{\O_X}\ar[r]^(.6)\mu\ar@{=}[d] & \E({\mathcal D}_{\bar
x})\ar@{>>}[d]\ar[r]& \T\ar[r]\ar[d]&0\\
0\ar[r] &(-1)^*_{\Pic^0 X}\widehat{\O_X}\ar[r]^(.6){\tilde\mu} &\widetilde{\E(\D_{\bar
x})}\ar[d]\ar[r]& \widetilde\T\ar[r]\ar[d]&0\\
&&0&0
} 
\end{equation}
If $\widetilde\T= 0$, then the middle horizontal short exact sequence splits. But, for $\alpha$ a
closed point in the support of $\T$ (by the previous claim we know that $\T\neq 0$),
$\mu\otimes\CC(\alpha)=0$ by item {\it (d)}, so $\mu$ cannot split. Therefore $\widetilde\T\neq 0$.

Let $e=\codim_{\Pic^0X} \supp\widetilde \T\geq 2$ (see item {\it (c)}). Then
$\codim_{\Pic^0X}\supp \cExt^e(\widetilde \T,\O_{\Pic^0X})=e$. Now, we apply the functor
$\cExt^i(\>\cdot\>,\O_{\Pic^0X})$ to the bottom row of \eqref{eq:tautilde} using Corollary
\ref{cor:duality}
\begin{equation*}
\ldots\to R^{e-1}\Phi_{P}\omega_X\to \cExt^e(\widetilde
\T,\O_{\Pic^0X})\to\cExt^e(\widetilde{\E(\D_{\bar x})},\O_{\Pic^0X})\to\ldots
\end{equation*}
Since $\widetilde{\E(\D_{\bar x})}$ is torsion-free, $\codim_{\Pic^0X}\supp
\cExt^e(\widetilde{\E(\D_{\bar x})},\O_{\Pic^0X})>e$. Therefore, we must have $\codim_{\Pic^0X}\supp
R^{e-1}\Phi_{P}\omega_X=e$ and $\gv(\omega_X)\leq 1$.
\end{proof}

\providecommand{\bysame}{\leavevmode\hbox
to3em{\hrulefill}\thinspace}


\begin{thebibliography}{BLNP}
\bibitem[BLNP]{BLNP} M.A. Barja, M. Lahoz, J.C. Naranjo and G. Pareschi, {\em On the bicanonical map
of irregular varieties}, preprint  arXiv:0907.4363. To appear in {\em J. Algebraic
Geom}.

\bibitem[BCP]{BCP}
I.~C. Bauer, F.~Catanese, and R.~Pignatelli.
\newblock Complex surfaces of general type: some recent progress.
\newblock In {\em Global aspects of complex geometry}, pages 1--58. Springer,
  Berlin, 2006.

\bibitem[Be]{Be}
A.~Beauville. 
\newblock Annulation du $H^1$ pour les fibr\'es en droites plats.
\newblock In {\em Complex algebraic varieties (Proc. Bayreuth 1990)}, pages 1--15; Springer, Berlin, 1992.


\bibitem[Bo]{Bom}
E.~Bombieri.
\newblock Canonical models of surfaces of general type.
\newblock {\em Inst. Hautes \'Etudes Sci. Publ. Math.}, (42):171--219, 1973.

\bibitem[CH1]{CHpm}
J.~A. Chen and C.~D. Hacon.
\newblock Pluricanonical maps of varieties of maximal {A}lbanese dimension.
\newblock {\em Math. Ann.}, 320(2):367--380, 2001.

\bibitem[CH2]{CHlinseries}
\bysame.
\newblock Linear series of irregular varieties.
\newblock In {\em Algebraic geometry in {E}ast {A}sia ({K}yoto, 2001)}, pages
  143--153. World Sci. Publ., River Edge, NJ, 2002.

\bibitem[CM]{CM}
C.~Ciliberto and M.~Mendes~Lopes.
\newblock On surfaces with {$p_g=q=2$} and non-birational bicanonical maps.
\newblock {\em Adv. Geom.}, 2(3):281--300, 2002.

\bibitem[Co]{conrad}
B.~Conrad.
\newblock {\em Grothendieck duality and base change}, volume 1750 of {\em
  Lecture Notes in Mathematics}.
\newblock Springer-Verlag, Berlin, 2000.

\bibitem[GL1]{GL1}
M.~Green and R.~Lazarsfeld.
\newblock Deformation theory, generic vanishing theorems, and some conjectures
  of {E}nriques, {C}atanese and {B}eauville.
\newblock {\em Invent. Math.}, 90(2):389--407, 1987.

\bibitem[GL2]{GL2}
\bysame.
\newblock Higher obstructions to deforming cohomology groups of line bundles.
\newblock {\em J. Amer. Math. Soc.}, 4(1):87--103, 1991.

\bibitem[Ha]{Happroach}
C.~D. Hacon.
\newblock A derived category approach to generic vanishing.
\newblock {\em J. Reine Angew. Math.}, 575:173--187, 2004.

\bibitem[La]{positivity}
R.~Lazarsfeld.
\newblock {\em Positivity in algebraic geometry {I} {\&} {II}}, volume 48 {\&}
  49 of {\em Ergebnisse der Mathematik und ihrer Grenzgebiete}.
\newblock Springer-Verlag, Berlin, 2004.
\newblock Positivity for vector bundles, and multiplier ideals.

\bibitem[LP]{LP}
R.~Lazarsfeld and M.~Popa.
\newblock {\em Derivative complex, BGG correspondence, and numerical inequalities for compact Kaehler manifolds.}
\newblock Invent. Math., 182 no.3 (2010), 605-633.


\bibitem[M]{Mdual}
S.~Mukai.
\newblock Duality between {$D(X)$} and {$D(\hat X)$} with its application to
  {P}icard sheaves.
\newblock {\em Nagoya Math. J.}, 81:153--175, 1981.

\bibitem[Mu]{MAV}
D.~Mumford.
\newblock {\em Abelian varieties}, volume~5 of {\em Tata Institute of
  Fundamental Research Studies in Mathematics}.
\newblock Published for the Tata Institute of Fundamental Research, Bombay,
  2008.
\newblock With appendices by C. P. Ramanujam and Yuri Manin, Corrected reprint
  of the second (1974) edition.

\bibitem[PP1]{PPreg1}
G.~Pareschi and M.~Popa.
\newblock Regularity on abelian varieties {I}.
\newblock {\em J. Amer. Math. Soc.}, 16(2):285--302, 2003.

\bibitem[PP2]{PPreg3}
\bysame.
\newblock Regularity on abelian varieties {III}: relationship with generic
  vanishing and applications.
\newblock {\em Clay Math. Proc.}, 2006.


\bibitem[PP3]{PPCdF}
\bysame.
\newblock Strong generic vanishing and a higher-dimensional {C}astelnuovo-de
  {F}ranchis inequality.
\newblock {\em Duke Math. J.}, 150(2):269--285, 2009.


\bibitem[PP4]{PPGVsheaves}
\bysame. 
{\em $GV$-sheaves, Fourier-Mukai transform, and Generic 
Vanishing}, preprint arXiv:math/0608127. To appear in {\em Amer. J. Math.}


\bibitem[S]{simpson}
C.~Simpson.
\newblock Subspaces of moduli spaces of rank one local systems.
\newblock {\em Ann. Sci. \'Ecole Norm. Sup. (4)}, 26(3):361--401, 1993.

\end{thebibliography}
\end{document}